\documentclass{article}
\usepackage[paper=a4paper,left=30mm,right=30mm,top=25mm,bottom=25mm]{geometry}
\usepackage[T1]{fontenc}
\usepackage{lmodern}
\usepackage{amsmath,amsthm,amssymb,mathtools}
\usepackage{graphicx,booktabs,enumerate}
\newcommand{\rhogammaroot}{}
\IfFileExists{rho_gamma_article/rho-gamma.tex}{%
\renewcommand{\rhogammaroot}{rho_gamma_article/}}{}
\IfFileExists{dsfont.sty}{\usepackage{dsfont}}{}

\usepackage[numbers,sort&compress]{natbib}
\usepackage{xcolor}
\definecolor{linkblue}{rgb}{0.10,0.25,0.45}
\usepackage[colorlinks=true,allcolors=linkblue]{hyperref}
\hypersetup{
	pdftitle={The exact region determined by Spearman's rho and Gini's gamma},
	pdfauthor={Jonathan Ansari, Marcus Rockel and Stefanie Steinma\ss{}l},
	pdfkeywords={Spearman's rho, Gini's gamma, attainable region, copula, Kantorovich duality, optimal transport}}

\definecolor{light-gray}{gray}{0.95}
\definecolor{darkblue}{rgb}{0,0,.5}
\definecolor{foxred}{rgb}{0.7, 0.11, 0.11}

\allowdisplaybreaks

\newtheorem{theorem}{Theorem}[section]
\newtheorem{proposition}[theorem]{Proposition}
\newtheorem{lemma}[theorem]{Lemma}
\newtheorem{corollary}[theorem]{Corollary}
\theoremstyle{definition}
\newtheorem{definition}[theorem]{Definition}
\newtheorem{example}[theorem]{Example}
\theoremstyle{remark}
\newtheorem{remark}[theorem]{Remark}

\newcommand{\de}{\,\mathrm{d}}
\newcommand{\E}{\mathbb{E}}
\newcommand{\PP}{\mathbb{P}}
\newcommand{\CC}{\mathcal{C}}
\newcommand{\couplings}{\mathcal{P}_\lambda}
\newcommand{\I}{[0,1]}
\newcommand{\Om}{\Omega_{\rho,\gamma}}
\newcommand{\rhoup}{\overline{\rho}}
\DeclareMathOperator{\sgn}{sgn}
\DeclareMathOperator{\Unif}{Unif}
\providecommand{\keywords}[1]{\noindent\textbf{Keywords } #1}

\title{\textbf{The exact region determined by Spearman's rho and Gini's gamma}}
\author{%
	Jonathan Ansari\thanks{Department of Mathematics, Paris Lodron Universit\"at Salzburg,
	Hellbrunner Stra\ss e 34, 5020 Salzburg, Austria.
	\texttt{jonathan.ansari@plus.ac.at} (Jonathan Ansari),
	\texttt{stefanie.steinmassl@plus.ac.at} (Stefanie Steinma\ss{}l)}
	\and
	Marcus Rockel\thanks{Department of Quantitative Finance, Institute for Economics,
	University of Freiburg, Rempartstr.\;16, 79098 Freiburg, Germany.
	\texttt{marcus.rockel@finance.uni-freiburg.de}}
	\and
	Stefanie Steinma\ss{}l\footnotemark[1]%
}
\date{\today}

\begin{document}
\maketitle

\begin{abstract}
We determine the exact attainable region of Spearman's rho and Gini's gamma over all bivariate copulas, resolving the remaining pairwise exact-region problem among Spearman's rho, Kendall's tau, Gini's gamma, Blomqvist's beta and Spearman's footrule. 
We give an explicit parametrization of the rho-maximal boundary for prescribed Gini's gamma and construct copulas attaining every boundary point. The boundary consists of an elementary arc up to a single junction and, beyond it, countably many algebraic pieces accumulating at comonotonicity. This description also gives the largest possible value of \(|\rho-\gamma|\) and the sharp thresholds beyond which rho and gamma must have the same sign. Our proof separates the signs and magnitudes of centred ranks, reducing the optimization of linear combinations of \(\rho\) and \(\gamma\) to an optimal transport problem for two uniform magnitudes. The extremizers combine a central antidiagonal block with rescaled auxiliary transport optimizers, and global optimality follows from a Kantorovich dual potential obtained by gluing the corresponding dual pieces.
\end{abstract}


\section{Introduction}\label{sec:intro}

Classical rank correlations quantify the degree of positive or negative dependence between two random variables. Their possible values are linked, although neither
their magnitudes nor their signs must coincide. To understand their interplay, it is important to determine their exact region, which provides all pairs of values that can be attained simultaneously. 
In this paper, we determine the exact region
\begin{equation}\label{eq:region-def}
    \Om\coloneqq\{(\rho(C),\gamma(C)):C\in\CC\}.
\end{equation}
between Spearman's rho and Gini's gamma. Our main result, Theorem \ref{thm:main} resolves the remaining problem of the pairwise comparison of classical rank correlations including these two coefficients as well as Kendall's tau, Blomqvist's beta,
and Spearman's footrule; see Table \ref{tab:literature} for an overview.

Let $\CC$ denote the class of bivariate copulas. By continuity and convexity arguments, the problem \eqref{eq:region-def} reduces to finding the
greatest Spearman's rho compatible with any fixed value of Gini's
gamma, i.e., to solve the maximization problem 
\begin{align}\label{def_optim_problem}
    \rhoup(g)\coloneqq\max\{\rho(C):C\in\CC,\ \gamma(C)=g\},
    \qquad -1\le g\le1.
\end{align}
Recall that for a bivariate copula \(C\), Spearman's rho and \(\gamma\) Gini's gamma are defined by
\begin{equation}\label{eq:intro-rank-moments}
    \rho(C)=\operatorname{Cor}(U,V)=12\,\E[UV]-3,
    \qquad
    \gamma(C)=2\,\E\bigl[|U+V-1|-|U-V|\bigr], \qquad (U,V)\sim C.
\end{equation}
Thus rho is the product-moment correlation between the uniform ranks, while
gamma contrasts their absolute deviations from the antidiagonal and the
diagonal. Written directly in terms of the copula, these quantities admit the representation
\begin{equation}\label{eq:defs}
\begin{aligned}
    \rho(C)&=12\int_{\I^2}C(u,v)\de u\de v-3, \quad
    \gamma(C)=4\int_0^1\bigl\{C(u,u)+C(u,1-u)\bigr\}\de u-2,
\end{aligned}
\end{equation}
see \citet[Section~5.1]{nelsen2006introduction} and
\citet[Section~2]{bukovsek2021spearman}.
The moment identities are recalled in Lemma~\ref{lem:representations}.
For the concordance interpretation of Gini's gamma, see
\citet{nelsen1998concordance}. The review by
\citet{genest2010spearman} provides further background on Gini's gamma and the related
Spearman's footrule, which we will use lateron.

Both Spearman's rho and Gini's gamma are measures of concordance in the sense of
\citet{scarsini1984measures}; see also
\citet[Section~2.4]{durante2015principles}. They take values in $[-1,1]$
and change sign under reflection of one coordinate. Further, write
$M(u,v)=\min(u,v)$ and $W(u,v)=\max(u+v-1,0)$ for the
Fr\'echet--Hoeffding copulas, and $\Pi(u,v)=uv$ for the independence copula.
Then, for these copulas, the coefficients pairs attain the values $(1,1)$, $(-1,-1)$ and $(0,0)$, respectively.

Of particular relevance for the present paper is the recent
solution of the rho--footrule problem by \citet{ansari2026exact}. Its optimal transport formulation and the corresponding dual potential provide a key ingredient in our analysis of the rho--gamma region. We will therefore adopt part of the notation introduced there when developing the boundary parametrization below.

\begin{table}[t]
\centering
\small
\caption{Representative exact attainable regions for classical concordance
coefficients and the principal methods used to derive them.}
\label{tab:literature}
\begin{tabular}{@{}p{0.18\linewidth}p{0.21\linewidth}p{0.54\linewidth}@{}}
\toprule
Coefficients & Reference & Main idea\\
\midrule
$(\tau,\rho)$
& \citep{schreyer2017exact}
& Combinatorial estimates and extremal shuffles of $M$.\\[2mm]

$(\beta,\phi)$, $(\beta,\gamma)$
& \citep{bukovsek2021spearman}
& Explicit attaining copulas combined with copula bounds.\\[2mm]

$(\phi,\gamma)$
& \citep{bukovsek2022exact}
& Explicit extremizers and sharp integral inequalities.\\[2mm]

$(\tau,\phi)$, $(\tau,\gamma)$
& \citep{bukovsek2023exact}
& Shuffles, approximation and integral inequalities.\\[2mm]

$(\rho,\phi)$
& \citep{ansari2026exact}
& An optimal-transport formulation with an explicit dual potential.\\[2mm]

$(\rho,\gamma)$
& This paper
& Sign--magnitude decomposition, optimal transport and a dual potential.\\
\bottomrule
\end{tabular}
\end{table}

We parametrize the upper boundary \(\overline{\rho}\) in \eqref{def_optim_problem} by a parameter $\theta\in[0,\infty]$,
with $C_0=M$ and $C_\infty=W$. The corresponding boundary point is written
as
$$
    \bigl(G(\theta),P(\theta)\bigr)
    =
    \bigl(\gamma(C_\theta),\rho(C_\theta)\bigr),
$$
where the family of copulas $C_\theta$ is introduced in
Definition~\ref{def:extremizers}. The explicit formulas for $G$ and $P$ are
given in \eqref{eq:param-boundary}; the auxiliary quantities entering these
formulas are defined in
\eqref{eq:large-parameters}--\eqref{eq:moments} and \eqref{eq:split}.

\begin{theorem}[The exact rho--gamma region]\label{thm:main}
The mapping $\theta\mapsto(G(\theta),P(\theta))$ is continuous on
$[0,\infty]$, while $G$ maps $[0,\infty]$ onto $[-1,1]$. For each
$\theta\in[0,\infty]$,
\begin{equation}\label{eq:maximum}
    P(\theta)
    =
    \max\{\rho(C):C\in\CC,\ \gamma(C)=G(\theta)\},
\end{equation}
and the copula $C_\theta$ from Definition~\ref{def:extremizers} realizes this
maximum. Consequently, $\rhoup(G(\theta))=P(\theta)$, and
\begin{equation}\label{eq:exact-region}
    \Om
    =
    \bigl\{(r,g):-1\le g\le1,\
    -\rhoup(-g)\le r\le\rhoup(g)\bigr\}.
\end{equation}
Moreover, the region is compact, convex and centrally symmetric, and
$\rhoup$ is continuous, concave and strictly increasing with
$\rhoup(\pm1)=\pm1$.
\end{theorem}

In Figure~\ref{fig:region}, gamma is placed on the horizontal axis and rho on
the vertical axis.
\begin{figure}[t!]
\centering
\includegraphics[width=0.72\linewidth]{\rhogammaroot 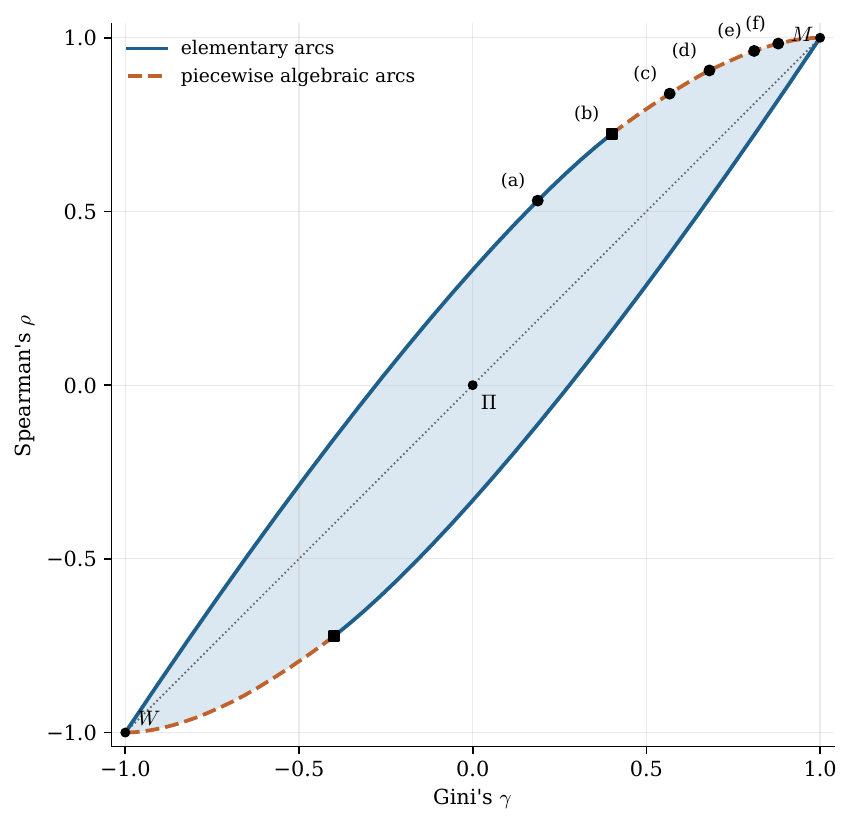}
\caption{The exact region $\Om$ from Theorem~\ref{thm:main}, with $\gamma$
on the horizontal axis and $\rho$ on the vertical axis, obtained by evaluating
\eqref{eq:param-boundary}. The elementary arcs are solid, and the countably
piecewise algebraic arcs are dashed. Squares indicate the junctions at
$\theta=1$ and its reflected point. The labels $M$, $W$ and $\Pi$ denote the
Fr\'echet--Hoeffding and independence copulas, while the dotted line represents
$\rho=\gamma$. The copulas shown in Figure~\ref{fig:extremizers} correspond to
the points (a)--(f).}
\label{fig:region}
\end{figure}
The upper boundary separates into two regimes. Set
\begin{equation}\label{eq:junction-intro}
    g_*
    =
    \frac{173-16\sqrt3}{363}
    =
    0.4002\ldots.
\end{equation}
For $-1\le g\le g_*$, the upper boundary curve \(\rhoup\) is given by the elementary formula
\begin{equation}\label{eq:elementary-intro}
    \rhoup(g)
    =
    \frac{144+420g+(9+65g)\sqrt{6-10g}}{500}.
\end{equation}
The parameter interval for this arc is $0\le\theta\le1$. Its extremizers are
shuffles of $M$ generated from the half-shift coupling, meaning the
deterministic relation $V_0=U_0+\tfrac12\pmod1$. Details are given in
Corollary~\ref{cor:elementary} and Example~\ref{ex:map}.

Within the rho--footrule problem, the half shift is the endpoint extremizer
with footrule value $-1/2$. Across the elementary regime, the normalized
coupling in each corner remains unchanged, and the construction varies only
through the relative dimensions of the central and corner blocks; compare
Figure~\ref{fig:extremizers}(a),(b). For the auxiliary distance problem
\eqref{eq:auxiliary-problem}, the largest admissible mean distance $1/2$
continues to minimize the objective, and the half shift attains it at constant
distance. Lemma~\ref{lem:halfshift-range} proves that its exact interval of
optimality is $0<\theta\le1$.

When $g_*<g<1$, equivalently $\theta>1$, the half shift is no longer optimal,
because reducing the mean distance improves the auxiliary objective. The
corner blocks then evolve through the nontrivial rho--footrule family from
\citet{ansari2026exact}, as displayed in
Figure~\ref{fig:extremizers}(c)--(f). The optimizer changes its form at
$\theta=N$ and $\theta=2N(N+1)/(2N+1)$, $N\in\mathbb N$. These transition
points produce countably many algebraic pieces converging to $(1,1)$. See
Remark~\ref{rem:anchors}.

The central reduction decomposes the centred ranks into signs and
magnitudes. Given $(U,V)\sim C$, define
\[
    X=2U-1,\qquad
    Y=2V-1,\qquad
    A=|X|,\qquad
    B=|Y|,\qquad
    S=\sgn(XY).
\]
The variables $A$ and $B$ are uniform on $\I$, and
Lemma~\ref{lem:representations} gives
\begin{equation}\label{eq:intro-signed-moments}
    \rho(C)=3\,\E[SAB],
    \qquad
    \gamma(C)=2\,\E[S\min(A,B)].
\end{equation}
It follows that, for every $t\ge0$,
\begin{equation}\label{eq:intro-sign-bound}
\begin{aligned}
    \rho(C)-\frac32t\,\gamma(C)
    &=
    3\,\E\!
    \left[
        S\min(A,B)\{\max(A,B)-t\}
    \right]\\
    &\le
    3\,\E
    \left[
        \min(A,B)|\max(A,B)-t|
    \right].
\end{aligned}
\end{equation}
Once a coupling of $(A,B)$ is fixed, equality results from taking
$S=1$ on $\{\max(A,B)\ge t\}$ and $S=-1$ otherwise.
Lemma~\ref{lem:sign-attainment} verifies that a copula with the chosen
magnitude law can realize these signs. Maximizing
\eqref{eq:intro-sign-bound} over all copulas therefore becomes
\begin{equation}\label{eq:intro-primal}
    \sup_{\pi\in\couplings}\int_{\I^2}
    \min(x,y)|\max(x,y)-t|\de\pi(x,y),
\end{equation}
where $\couplings$ denotes the probability measures on $\I^2$ whose two
marginals are the uniform probability measure $\lambda$.

The construction is organized around a splitting point $a=a_\theta$. On the
interval below $a$, the magnitudes agree but the signs are opposite, which
creates a central antidiagonal segment in the rank square. Above $a$, the
signs agree and the magnitudes follow a rescaled rho--footrule optimizer. The
result is an ordinal sum formed by a central $W$-block and two corner blocks,
with a reflection of both coordinates in the lower corner.

Global optimality is established by joining two dual potentials at $a$: the
quadratic $x(t-x)/2$ on $[0,a]$ and a rescaled rho--footrule potential,
combined with another quadratic term, on $[a,1]$. The defining choice of
$a$ makes the two expressions coincide. The essential point is to verify the
dual inequality over the whole square, especially for pairs whose coordinates
lie on opposite sides of $a$, even though the proposed coupling assigns those
pairs no mass. Equality on the support, together with the optimal assignment
of signs, then establishes the sharp supporting relation
\begin{equation}\label{eq:intro-support}
    \rho(C)-\frac32t_\theta\gamma(C)
    \le
    P(\theta)-\frac32t_\theta G(\theta),
    \qquad C\in\CC,
\end{equation}
with equality when $C=C_\theta$. Letting $\theta$ vary traces the complete
rho-maximal boundary in Theorem~\ref{thm:main}. Reflection supplies the
rho-minimal boundary, and convexity gives all intervening points.

Two sharp comparisons follow from the region. Proposition~\ref{prop:discrepancy}
establishes
\begin{equation}\label{eq:intro-discrepancy}
    \max_{C\in\CC}|\rho(C)-\gamma(C)|
    =\frac{1064+160\sqrt{10}}{4563}
    =0.34406\ldots,
\end{equation}
and this value is attained by a shuffle on the elementary arc and by its
reflection. Corollary~\ref{cor:signs} supplies the optimal sign thresholds:
$|\rho(C)|>0.33209\ldots$ or
$|\gamma(C)|>0.27858\ldots$ guarantees that rho and gamma have the same
sign. Formula \eqref{eq:thresholds} gives the constants exactly.

Section~\ref{sec:main} presents the boundary parametrization and the ensuing
comparisons. In Section~\ref{sec:construction}, we derive the sign--magnitude
reduction and build the extremal copulas. Section~\ref{sec:proof} establishes
the global dual inequality, proves Theorem~\ref{thm:main} and derives its
consequences.

\section{Boundary parametrization and consequences}\label{sec:main}

Three ingredients determine the parametrization: the first and second
moments of the distance under a rho--footrule optimizer, together with the
endpoint evaluation of its dual potential. We record these data first and
then derive the scale at which the optimizer enters the extremal rho--gamma
copula.

\subsection{The rho--footrule input and the two parameter regimes}
\label{sec:parametrization}

Fix $\theta>0$ and set $s=1/\theta$. In what follows, $s$ is the coefficient
of the distance term in the auxiliary rho--footrule problem
\begin{equation}\label{eq:auxiliary-problem}
\min_{\pi\in\couplings}
\E_\pi\bigl[|U_0-V_0|^2-s|U_0-V_0|\bigr].
\end{equation}
Denote the optimizer and dual potential from Lemma~\ref{lem:input} by
$(\pi_\theta,h_\theta)$, and define
\[
m_\theta=\E_{\pi_\theta}|U_0-V_0|,
\qquad
q_\theta=\E_{\pi_\theta}|U_0-V_0|^2.
\]
We use $c_\theta=h_\theta(0)$ for the value of the potential at the endpoint.

When $0<\theta\le1$, the optimizer takes the half-shift form
\[
V_0=U_0+\frac12\pmod1,
\]
as established in Lemma~\ref{lem:halfshift-range}. It follows that
\begin{equation}\label{eq:large-parameters}
m_\theta=\frac12,\qquad
q_\theta=\frac14,\qquad
c_\theta=\frac38-\frac s2.
\end{equation}
Because the distance moments remain constant throughout this range, they
produce the elementary branch in Corollary~\ref{cor:elementary}.

Suppose instead that $\theta>1$. The solution of
\eqref{eq:auxiliary-problem} is then the nontrivial rho--footrule optimizer
constructed by \citet{ansari2026exact}. Introduce
\[
N=\lfloor\theta\rfloor,\qquad
L=\frac1{2N},\qquad
R=\frac1{2(N+1)},
\]
for which $2R<s\le2L$, and set
\begin{equation}\label{eq:small-parameters}
\ell=
\begin{cases}
L,&s\ge L+R,\\
R,&s<L+R,
\end{cases}
\qquad
\delta=s-2\ell,
\qquad
p=1-2N(N+1)|\delta|.
\end{equation}
These quantities yield
\begin{equation}\label{eq:moments}
\begin{aligned}
m_\theta&=\ell+N(N+1)\delta|\delta|,\\
q_\theta&=\ell(2m_\theta-\ell)
          +\frac23N(N+1)|\delta|^3,\\
c_\theta&=\frac{\ell^2-s\ell-\ell p\delta}{2}.
\end{aligned}
\end{equation}
Under $\pi_\theta$, the distance distribution places mass $p$ at $\ell$ and
has density $2N(N+1)$ on the interval with endpoints $\ell$ and
$\ell+\delta$. Consequently, $|\delta|\le L-R$. The integer values
$\theta=N$ are therefore the distinguished points where $\delta=0$. A second
family of junctions occurs at
\[
\theta=\frac1{L+R}=\frac{2N(N+1)}{2N+1}
\]
where $p=0$ and the two adjacent formulas coincide. Since these junctions
converge to $\theta=\infty$, the boundary segment tending to $(1,1)$ splits
into countably many algebraic arcs.

\subsection{From the corner optimizer to the rho--gamma boundary}

Write $a_\theta$ for both the length and the mass of the central magnitude
interval, and let $z_\theta=1-a_\theta$. The supporting parameter is
$t_\theta=s z_\theta=z_\theta/\theta$. We normalize the central length by
$\alpha_\theta=a_\theta/t_\theta$. Hence
$a_\theta=\alpha_\theta t_\theta$ and $z_\theta=\theta t_\theta$, while the
dual-potential matching condition from Section~\ref{sec:gluing} becomes
\[
\alpha_\theta^2-\alpha_\theta=\frac{\theta^2c_\theta}{2}.
\]
Taking the larger solution gives
\begin{equation}\label{eq:split}
\alpha_\theta=\frac{1+\sqrt{1+2\theta^2c_\theta}}{2},
\qquad
t_\theta=\frac1{\theta+\alpha_\theta},
\qquad
a_\theta=\alpha_\theta t_\theta,
\qquad
z_\theta=\theta t_\theta.
\end{equation}
By Lemma~\ref{lem:parameter-checks}, the expression under the square root is
positive and \(t_\theta/2\le a_\theta<t_\theta\). These bounds provide the
global dual feasibility needed in Lemma~\ref{lem:extension}.

The resulting boundary coordinates are
\begin{equation}\label{eq:param-boundary}
G(\theta)
=
1-2a_\theta^2-z_\theta^2m_\theta,
\qquad
P(\theta)
=
1-2a_\theta^3-\frac32z_\theta^3q_\theta,
\end{equation}
subject to the endpoint conventions
\[
(G(0),P(0))=(-1,-1),
\qquad
(G(\infty),P(\infty))=(1,1).
\]
Proposition~\ref{prop:moments} gives $G(\theta)=\gamma(C_\theta)$ and
$P(\theta)=\rho(C_\theta)$, and Theorem~\ref{thm:main} places
$(P(\theta),G(\theta))$ on the rho-maximal boundary.

The three-block ordinal sum $C_\theta$ is displayed in
Figure~\ref{fig:extremizers}. Its central $W$-block has side length
$a_\theta$, whereas each corner has side length $z_\theta/2$ and carries a
rescaled version of $\pi_\theta$. In the lower corner, both coordinates are
reflected. The relation $t_\theta=s z_\theta$ explains why the linear and
quadratic distance contributions involve different powers of $z_\theta$.

\begin{figure}[t!]
\centering
\includegraphics[width=\linewidth]{\rhogammaroot 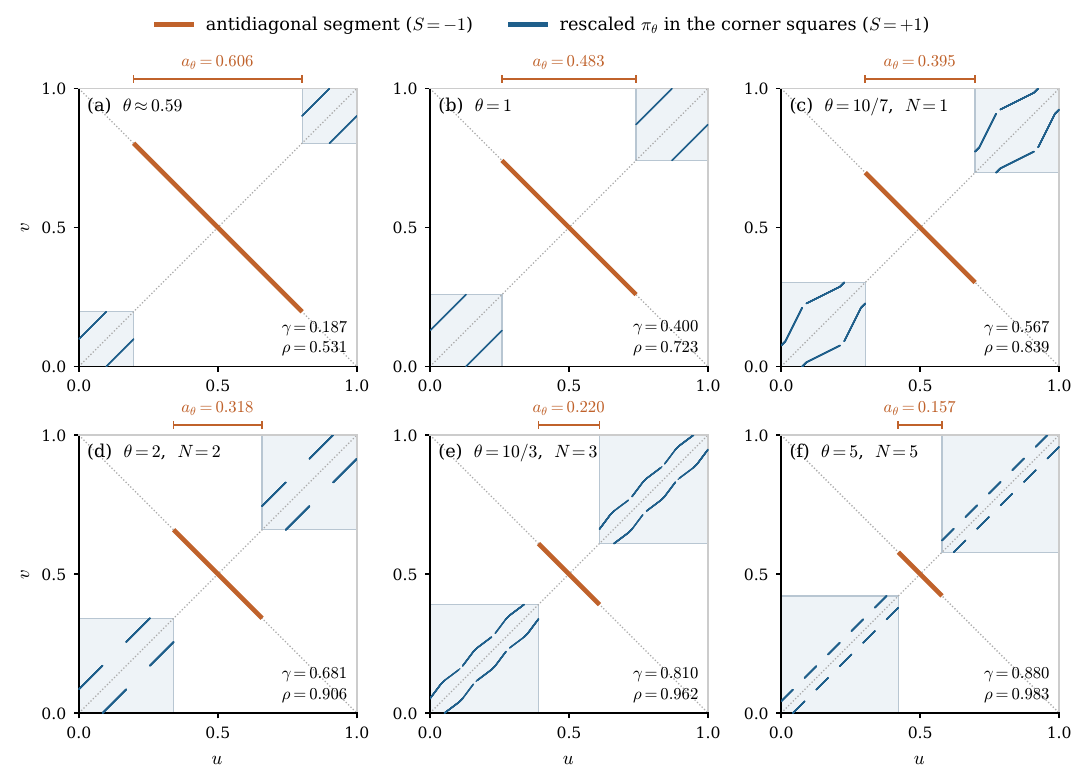}
\caption{Support sets of $C_\theta$ at the points labeled (a)--(f) in
Figure~\ref{fig:region}. The orange segment spans $a_\theta$ in each
coordinate, has mass $a_\theta$, and corresponds to opposite rank signs
($S=-1$). Each blue corner has side length and mass $z_\theta/2$ and
corresponds to equal rank signs ($S=1$). Simultaneous reflection of both
coordinates maps the upper corner to the lower one. In panels (a),(b), the
corner coupling is the fixed half shift. Panels (c)--(f) depict the varying
rho--footrule corner coupling using assignment approximations to $\pi_\theta$
on 2000 midpoint ranks, where $N=\lfloor\theta\rfloor$. The maximal
discrepancy occurs in panel (a), and panel (b) indicates the junction
$\theta=1$.}
\label{fig:extremizers}
\end{figure}

For the elementary branch, define
\begin{equation}\label{eq:junction}
d_*=\frac{12-2\sqrt3}{33}=0.2586\ldots,
\qquad
g_*=\frac{173-16\sqrt3}{363}=0.4002\ldots.
\end{equation}

\begin{samepage}
\begin{corollary}[The elementary arc]\label{cor:elementary}
If $-1\le g\le g_*$, then
\begin{equation}\label{eq:elementary}
\rhoup(g)
=
\frac{144+420g+(9+65g)\sqrt{6-10g}}{500}.
\end{equation}
Equivalently, the same portion of the upper boundary has the parametrization
\begin{equation}\label{eq:d-arc}
g=-1+8d-10d^2,
\qquad
r=-1+12d-24d^2+13d^3,
\qquad
0\le d\le d_*.
\end{equation}
For every $d>0$, Example~\ref{ex:map} gives a five-piece shuffle of $M$
that attains the corresponding point. The endpoint $d=0$ is attained by $W$.
\end{corollary}
\end{samepage}

\begin{proof}
Assume $0<\theta\le1$. Substituting \eqref{eq:large-parameters} into
\eqref{eq:split}--\eqref{eq:param-boundary} and setting
$d=z_\theta/2$, so that $a_\theta=1-2d$, yields
\[
G(\theta)=1-2(1-2d)^2-2d^2,
\qquad
P(\theta)=1-2(1-2d)^3-3d^3,
\]
which agrees with \eqref{eq:d-arc}. We also have
\[
\frac{\alpha_\theta}{\theta}
=\frac{s+\sqrt{s^2-s+3/4}}2.
\]
As $\theta$ runs from $0$ to $1$, this ratio decreases continuously from
infinity to $(2+\sqrt3)/4$. Therefore
$d=\theta/(2(\theta+\alpha_\theta))$ grows from $0$ to $d_*$. On
$[0,d_*]$, the first polynomial in \eqref{eq:d-arc} is strictly increasing,
and its inverse is
\[
d=\frac{4-\sqrt{6-10g}}{10}.
\]
Inserting this expression into the second polynomial proves
\eqref{eq:elementary}. Example~\ref{ex:map} together with
Theorem~\ref{thm:main} establishes attainment.
\end{proof}

\begin{remark}[The arc towards $(1,1)$]\label{rem:anchors}
For $\theta>1$, the formulas have breakpoints at
\[
\theta=N
\qquad\text{and}\qquad
\theta=\frac{2N(N+1)}{2N+1},
\qquad N\in\mathbb N.
\]
At an integer value $\theta=N$, we have $\delta=0$ and
$\alpha_N=\kappa\coloneqq(2+\sqrt3)/4$, and therefore
\[
G(N)
=
1-\frac{2\kappa^2+N/2}{(N+\kappa)^2},
\qquad
P(N)
=
1-\frac{2\kappa^3+3N/8}{(N+\kappa)^3}.
\]
This sequence tends to $(1,1)$. Between any two consecutive junctions, the
formulas use only rational operations and a single square root. Hence each
boundary piece is algebraic.

Near comonotonicity, the boundary satisfies
\begin{equation}\label{eq:asymptotic}
\rhoup(g)
=
1-\frac32(1-g)^2+O((1-g)^3),
\qquad g\uparrow1.
\end{equation}
Indeed, $\theta\to\infty$ implies $s\downarrow0$, and throughout all pieces
$|\delta|=O(s^2)$ uniformly. Consequently,
$m_\theta=s/2+O(s^2)$, $q_\theta=s^2/4+O(s^3)$, and $c_\theta=O(s^2)$.
Because $\alpha_\theta$ stays bounded, \eqref{eq:split} then yields
$a_\theta=O(s)$ and $z_\theta=1+O(s)$. Thus
$1-G(\theta)=s/2+O(s^2)$ and $1-P(\theta)=3s^2/8+O(s^3)$.
\end{remark}

\subsection{Two consequences}

We first identify the greatest possible separation between rho and gamma.

\begin{proposition}[Largest discrepancy]\label{prop:discrepancy}
Define
\begin{equation}\label{eq:discrepancy-constant}
d_0=\frac{14-2\sqrt{10}}{39}=0.1968\dots,
\qquad
\Delta_0
=
\frac{1064+160\sqrt{10}}{4563}
=
0.34406\ldots.
\end{equation}
Then
\begin{equation}\label{eq:discrepancy}
\max_{C\in\CC}|\rho(C)-\gamma(C)|=\Delta_0.
\end{equation}
The largest value of $\rho-\gamma$ occurs on the elementary boundary at
$d=d_0$, at the point
\[
(\rho,\gamma)
=
(0.5311\ldots,\,0.1871\ldots),
\]
while its reflection attains the largest value of $\gamma-\rho$.
\end{proposition}

We next give the exact range of either coefficient subject to the other
being zero.

\begin{corollary}[Sign thresholds]\label{cor:signs}
Set
\begin{equation}\label{eq:thresholds}
r_0=\frac{144+9\sqrt6}{500}=0.33209\ldots,
\qquad
g_0=\frac{72-3\sqrt{69}}{169}=0.27858\ldots.
\end{equation}
Then
\begin{equation}\label{eq:zero-sections}
\{\rho(C):\gamma(C)=0\}=[-r_0,r_0],
\qquad
\{\gamma(C):\rho(C)=0\}=[-g_0,g_0].
\end{equation}
It follows that $\gamma(C)$ and $\rho(C)$ share a sign whenever
$|\rho(C)|>r_0$, and that $\rho(C)$ and $\gamma(C)$ share a sign whenever
$|\gamma(C)|>g_0$. Both thresholds are sharp.
\end{corollary}

Proofs of Proposition~\ref{prop:discrepancy} and
Corollary~\ref{cor:signs} appear in
Section~\ref{sec:consequence-proofs}.
\section{Signed magnitudes and extremal copulas}\label{sec:construction}

This section reduces the construction to a transport problem for magnitudes,
summarizes the rho--footrule ingredients, and builds the boundary-attaining copulas.

\subsection{Separating signs and magnitudes}\label{sec:signs}

The next representation decomposes centred ranks into their signs and
magnitudes. The converse statement ensures that these two components may be
optimized independently while preserving uniform rank margins.

\begin{lemma}[Moment representations]\label{lem:representations}
	Suppose $(U,V)\sim C\in\CC$ and define $X=2U-1$, $Y=2V-1$, $A=|X|$, $B=|Y|$ and $S=\sgn(XY)$.
	The variables $A$ and $B$ are uniformly distributed on $(0,1)$, and
	\begin{equation}\label{eq:rank-moments}
		\rho(C)=12\,\E[UV]-3=1-6\,\E(U-V)^2,\qquad
		\gamma(C)=2\,\E\bigl[|U+V-1|-|U-V|\bigr],
	\end{equation}
	\begin{equation}\label{eq:signed-moments}
		\rho(C)=3\,\E[SAB],\qquad
		\gamma(C)=2\,\E\bigl[S\min(A,B)\bigr].
	\end{equation}
	Conversely, for every random vector $(A,B,S)$ with $A,B\sim\Unif(0,1)$ and $S\in\{-1,1\}$, there exists a copula $C$ such that, for $(U,V)\sim C$, the magnitude pair $(|2U-1|,|2V-1|)$ and the sign $\sgn((2U-1)(2V-1))$ jointly have the law prescribed by $(A,B,S)$. Consequently, $C$ obeys \eqref{eq:signed-moments} for this triple.
\end{lemma}

\begin{proof}
	The standard formula $\rho(C)=12\,\E[UV]-3$ appears in \citet[Sec.~5.1]{nelsen2006introduction}, and $\E U^2=\E V^2=1/3$ converts it into the second expression for $\rho$ in \eqref{eq:rank-moments}.
	Applying Fubini's theorem to $\gamma$ yields
	\[
		\int_0^1C(u,u)\de u=\E[1-\max(U,V)],\qquad\int_0^1C(u,1-u)\de u=\E[(1-U-V)_+].
	\]
	Combining these formulas with $|U-V|=2\max(U,V)-U-V$, $|U+V-1|=2(1-U-V)_++U+V-1$ and $\E(U+V-1)=0$ proves the asserted expression for $\gamma$ in \eqref{eq:rank-moments}.
	After centring, we have $12\,\E[UV]-3=3\,\E[XY]$ and $2\,\E[|U+V-1|-|U-V|]=\E[|X+Y|-|X-Y|]$. The relation
	\[
		|X+Y|-|X-Y|=2\sgn(XY)\min(|X|,|Y|)
	\]
	therefore proves \eqref{eq:signed-moments}, while the uniformity of $A$ and $B$ follows because both $X$ and $Y$ are uniform on $[-1,1]$.

	To prove the converse, introduce a fair random sign $\varepsilon$ independent of $(A,B,S)$, and define $X=\varepsilon A$ and $Y=\varepsilon SB$.
	Conditional on $(A,B,S)$, the signs $\varepsilon$ and $\varepsilon S$ are each fair, so every uniform magnitude is placed on the positive and negative half-intervals with equal probability.
	It follows that $X$ and $Y$ are uniform on $[-1,1]$, with $|X|=A$, $|Y|=B$ and $\sgn(XY)=S$. Thus the copula of $\bigl((1+X)/2,(1+Y)/2\bigr)$ has the claimed moments.
\end{proof}

For $t\ge0$ and $x,y\in\I$, define $c_t(x,y)\coloneqq\min(x,y)\,|\max(x,y)-t|$.
Combining Lemma~\ref{lem:representations} with $AB-t\min(A,B)=\min(A,B)\{\max(A,B)-t\}$ gives the supporting functional
\begin{equation}\label{eq:signed-objective}
	\rho(C)-\tfrac32t\,\gamma(C)=3\,\E\bigl[S\min(A,B)\{\max(A,B)-t\}\bigr]\le3\,\E c_t(A,B),
\end{equation}
and equality occurs precisely when $S=\sgn(\max(A,B)-t)$ almost surely on $\{c_t(A,B)>0\}$.
Thus the sign is optimized pointwise, and the only remaining optimization concerns the coupling between the two uniform magnitudes.

\begin{lemma}[Attainability of the sign bound]\label{lem:sign-attainment}
Given $t\ge0$ and any coupling $(A,B)$ of uniform magnitudes, one can find a copula with this magnitude distribution that attains equality in \eqref{eq:intro-sign-bound}.
\end{lemma}

\begin{proof}
Set $S=1$ on $\{\max(A,B)\ge t\}$ and $S=-1$ on its complement, and invoke the converse part of Lemma~\ref{lem:representations}.
The two margins remain uniform, while $S\{\max(A,B)-t\}=|\max(A,B)-t|$ holds almost surely.
\end{proof}

\subsection{The magnitude transport problem and its dual}\label{sec:duality}

Every $\pi\in\couplings$ determines a copula through its distribution function, and each copula likewise gives an element of $\couplings$.
For $t>0$, introduce the transport value
\begin{equation}\label{eq:primal}
\mathsf V(t)\coloneqq\max_{\pi\in\couplings}\int_{\I^2}c_t(x,y)\de\pi(x,y).
\end{equation}
The maximum is attained because $\couplings$ is weakly compact and $c_t$ is continuous. See \citet[Thm.~4.1 and its proof]{villani2009optimal}, with the theorem applied to the minimization cost $-c_t$.
It follows from \eqref{eq:signed-objective} and Lemma~\ref{lem:sign-attainment} that
\begin{equation}\label{eq:support-function}
	\max_{C\in\CC}\bigl\{\rho(C)-\tfrac32t\,\gamma(C)\bigr\}=3\,\mathsf V(t),
\end{equation}
so solving \eqref{eq:primal} determines the support function of $\Om$ in direction $(1,-\tfrac32t)$.
When the axes are ordered as $(\gamma,\rho)$, an upper bound $b$ on this functional becomes the line $\rho\le b+3t\gamma/2$.
Equality for a copula means that this line supports the region at the corresponding coefficient pair.
In the maximization convention, \citet[Thm.~5.10]{villani2009optimal} supplies the Kantorovich dual of \eqref{eq:primal}. The symmetry of $c_t$ reduces it to
\begin{equation}\label{eq:dual}
	\mathsf D(t)\coloneqq\inf\Bigl\{2\int_0^1f(x)\de x:\ f\in C(\I),\ f(x)+f(y)\ge c_t(x,y)\ \text{for all }(x,y)\in\I^2\Bigr\}.
\end{equation}
We call any continuous $f$ obeying the constraint in \eqref{eq:dual} a \emph{feasible potential}, and associate with it the set
\[
	\Gamma_f\coloneqq\bigl\{(x,y)\in\I^2:f(x)+f(y)=c_t(x,y)\bigr\}
\]
called its \emph{contact set}.
It suffices to use one potential. Indeed, if continuous functions $\varphi$ and $\psi$ satisfy $\varphi(x)+\psi(y)\ge c_t(x,y)$, then symmetry of $c_t$ makes $f=(\varphi+\psi)/2$ feasible, and $2\int_0^1f(x)\de x=\int_0^1\varphi(x)\de x+\int_0^1\psi(y)\de y$.

\begin{lemma}[Weak duality and optimality via the contact set]\label{lem:duality}
	Fix $t>0$.
	\begin{enumerate}[(i)]
		\item For any feasible potential $f$ and any $\pi\in\couplings$, one has $\int_{\I^2}c_t(x,y)\de\pi(x,y)\le2\int_0^1f(x)\de x$. Hence $\mathsf V(t)\le\mathsf D(t)$.
		\item Suppose that $\pi\in\couplings$ is supported on the contact set $\Gamma_f$ of a feasible potential $f$. Then $\pi$ and $f$ solve \eqref{eq:primal} and \eqref{eq:dual}, respectively, and $\mathsf V(t)=\mathsf D(t)=2\int_0^1f(x)\de x$.
	\end{enumerate}
\end{lemma}

\begin{proof}
	Uniformity of both marginals of $\pi$ gives $\int_{\I^2}\{f(x)+f(y)\}\de\pi(x,y)=2\int_0^1f(x)\de x$. Integrating the constraint in \eqref{eq:dual} against $\pi$ now proves (i).
	When $\pi(\I^2\setminus\Gamma_f)=0$, the corresponding integral is an equality, so $\int_{\I^2}c_t(x,y)\de\pi(x,y)=2\int_0^1f(x)\de x$. Part~(i) then yields $\mathsf V(t)\ge2\int_0^1f(x)\de x\ge\mathsf D(t)\ge\mathsf V(t)$, forcing equality throughout.
\end{proof}

\begin{remark}[Strong duality]\label{rem:strong-duality}
	Continuity of $c_t$ on the compact square allows the Kantorovich duality theorem, see \citet[Thm.~5.10]{villani2009optimal}, to give $\mathsf V(t)=\mathsf D(t)$ for all $t>0$. The preceding symmetrization then yields the same conclusion for the one-potential formulation \eqref{eq:dual}.
	At $t=t_\theta$, we instead verify the equality directly by pairing a feasible potential with a coupling carried by its contact set and applying Lemma~\ref{lem:duality}(ii).
	In the rho--footrule problem studied by \citet{ansari2026exact}, a linear moment is held fixed and constrained Kantorovich duality is used, see \citet[Thm.~2.1]{zaev2015monge}. The present formulation fixes the multiplier $t$ and incorporates the gamma constraint into the cost, which leaves only unconstrained duality.
\end{remark}

\subsection{The auxiliary rho--footrule problem}\label{sec:input}

We need the primal--dual solution of \eqref{eq:auxiliary-problem} in the
form stated next. The case $\theta>1$ comes from \citet{ansari2026exact},
and every numbered reference to that work below concerns
arXiv:2608.20176v1. For $0<\theta\le1$, we establish the half-shift solution
here. The later gluing construction uses four features: the dual inequality,
equality on the optimizer, normalization at the endpoint, and the derivative
estimate included in the lemma.
The potential has Lipschitz constant
\begin{equation}\label{eq:lipschitz-parameter}
w_\theta=\begin{cases}s-1,&0<\theta\le1,\\ |\delta|,&\theta>1.\end{cases}
\end{equation}

\begin{lemma}[The rho--footrule optimizer and its potential]\label{lem:input}
	For each $\theta>0$, there are a coupling $\pi_\theta$ of uniform random variables $(U_0,V_0)$ on $\I$ and a Lipschitz function $h_\theta\colon\I\to\mathbb{R}$ for which
	\begin{equation}\label{eq:input-moments}
		\E|U_0-V_0|=m_\theta,\qquad\E|U_0-V_0|^2=q_\theta,
	\end{equation}
	\begin{equation}\label{eq:input-dual}
		h_\theta(u)+h_\theta(v)\le(u-v)^2-s|u-v|\quad\text{for all }u,v\in\I,
		\qquad\text{with equality $\pi_\theta$-almost surely,}
	\end{equation}
	and
	\begin{equation}\label{eq:input-endpoint}
		h_\theta(0)=c_\theta,\qquad|h_\theta'|\le w_\theta\ \text{almost everywhere},
	\end{equation}
	with $m_\theta,q_\theta,c_\theta$ specified by \eqref{eq:large-parameters}--\eqref{eq:moments} and $w_\theta$ defined in \eqref{eq:lipschitz-parameter}.
\end{lemma}

\begin{proof}
	Begin with $\theta>1$.
	The optimizer in \citet[Sec.~4]{ansari2026exact} is indexed by the footrule value $x$. If $N$, $L$, and $R$ are chosen as in \eqref{eq:small-parameters}, then the multiplier called $\theta(x)$ in that paper equals the present $s=1/\theta$: on the left half of the $N$-th footrule interval it is $s=2L-v$, and on the right half it is $s=2R+v$, where $v\in[0,L-R]$ and $p=1-2N(N+1)v$.
	Consequently, each $\theta\in(1,\infty)$ determines a unique $x$, with $v=|\delta|$ and with $\ell=L$ or $\ell=R$ according to the branch in \eqref{eq:small-parameters}. At $s=L+R$, formula \eqref{eq:small-parameters} chooses $\ell=L$ although $x$ belongs to the right half. At that point, $p=0$ on both branches, so they produce an identical coupling and potential.
	Take $\pi_\theta$ to be the coupling $\pi_x$ defined in that work. Its marginals are uniform by Lemmas~3.3 and~4.3 therein.
	According to Lemma~4.5 of that paper, $|U_0-V_0|$ under $\pi_\theta$ has distribution
	\[
		p\,\delta_\ell+2N(N+1)\,\lambda\big|_{[\min(\ell,\ell+\delta),\,\max(\ell,\ell+\delta)]},
	\]
	where $\delta_\ell$ is the unit point mass at $\ell$.
	The continuous component has mass $2N(N+1)|\delta|=1-p$. Integrating the first two powers of the distance against this law yields \eqref{eq:input-moments}, with the values in \eqref{eq:moments}.
	Now choose as $h_\theta$ the potential $g_x$ from Definition~5.1 of that paper.
	Its endpoint value is $g_x(0)=c_x$, with $c_x=c_\theta$ as in \eqref{eq:moments}. Lemma~5.3 shows that it is $s$-periodic and continuously differentiable, with a derivative whose absolute value is at most $v$. Lemmas~5.4 and~5.5 establish equality in \eqref{eq:input-dual} $\pi_\theta$-almost surely and the inequality on $\{u\le v\}$, which extends to $\I^2$ by symmetry.

	For $0<\theta\le1$, use the half shift $V_0=U_0+\tfrac12\pmod1$. Since $|U_0-V_0|=\tfrac12$ almost surely, \eqref{eq:input-moments} follows with the parameters in \eqref{eq:large-parameters}. Define
	\begin{equation}\label{eq:halfshift-dual}
		h_\theta(u)=\tfrac38-\tfrac s2+(s-1)\min(u,1-u).
	\end{equation}
	This gives $h_\theta(0)=c_\theta$. Moreover, $|h_\theta'|=s-1=w_\theta$.
	Writing $d=|u-v|$, the terms $\min(u,1-u)$ and $\min(v,1-v)$ sum to at most $1-d$, and therefore
	\[
		h_\theta(u)+h_\theta(v)\le\tfrac34-s+(s-1)(1-d)=d^2-sd-(d-\tfrac12)^2\le d^2-sd.
	\]
	For the half shift, $d=\tfrac12$ and $\min(u,1-u)+\min(v,1-v)=\tfrac12$, making both inequalities equalities.
\end{proof}

\begin{lemma}[The half-shift regime]\label{lem:halfshift-range}
For $\theta>0$, the half shift $V_0=U_0+1/2\pmod1$ is an optimizer of \eqref{eq:auxiliary-problem} exactly when $0<\theta\le1$.
\end{lemma}

\begin{proof}
When $0<\theta\le1$, optimality follows from the potential \eqref{eq:halfshift-dual} and its equality on the half shift, established in Lemma~\ref{lem:input}.
Now suppose $\theta>1$, so $s<1$. Select $0<\varepsilon<\min\{1/4,1-s\}$
and set $\vartheta=1/(1-\varepsilon)$. For the coupling $\pi_\vartheta$ supplied by
Lemma~\ref{lem:input}, we have $N=1$, $\ell=1/2$ and $\delta=-\varepsilon$, and hence
\[
m_\vartheta=\tfrac12-2\varepsilon^2,\qquad
q_\vartheta=\tfrac14-2\varepsilon^2+\tfrac43\varepsilon^3.
\]
At multiplier $s$, the value for $\pi_\vartheta$ minus the half-shift value is
\[
q_\vartheta-s m_\vartheta-(\tfrac14-\tfrac s2)
=-2(1-s)\varepsilon^2+\tfrac43\varepsilon^3<0.
\]
Hence the half shift cannot be optimal when $\theta>1$.
\end{proof}

The auxiliary optimizer therefore changes precisely at $\theta=1$, which
corresponds to the junction $g_*=G(1)$. The next estimates constrain the
splitting parameters in each regime.

\begin{lemma}[Parameter inequalities]\label{lem:parameter-checks}
	For all $\theta>0$, the inequalities
	\[
		c_\theta<0,\qquad1+2\theta^2c_\theta\ge\theta^2w_\theta^2,\qquad
		\frac{1+\theta w_\theta}{2}\le\alpha_\theta<1,
	\]
	hold and imply
	\begin{equation}\label{eq:split-bounds}
		\frac{t_\theta}{2}\le a_\theta<t_\theta.
	\end{equation}
\end{lemma}

\begin{proof}
	First consider the branch $\ell=L$, and write $v=-\delta\in[0,L-R]$. Then $s=2L-v$, $p=1-2N(N+1)v\in[0,1]$, and $v\le L/(N+1)$.
	In this notation,
	\[
		2c_\theta=-L^2+Lv(1+p),\qquad
		s^2+2c_\theta-v^2=3L^2-Lv(3-p)\ge3L^2-\frac{3L^2}{N+1}>0.
	\]
	We also have $v(1+p)<L$. For $N\ge2$, this follows from $v(1+p)\le2v\le2L/(N+1)$. When $N=1$, it follows from $v(1+p)=v(2-4v)\le\tfrac14<L$.
	On the branch $\ell=R$, set $v=\delta\in[0,L-R)$. Here $s=2R+v$ and $p\in(0,1]$, while
	\[
		2c_\theta=-R^2-Rv(1+p)<0,\qquad
		s^2+2c_\theta-v^2=3R^2+Rv(3-p)>0.
	\]
	If $0<\theta\le1$, formula \eqref{eq:large-parameters} instead yields $c_\theta<0$ and $s^2+2c_\theta-(s-1)^2=s-\tfrac14>0$.
	Multiplying $s^2+2c_\theta\ge w_\theta^2$ by $\theta^2$ proves the middle inequality in the statement.
	The estimates for $\alpha_\theta$ are immediate from \eqref{eq:split}, together with $c_\theta<0$.
	The relation $a_\theta=\alpha_\theta t_\theta$ then gives \eqref{eq:split-bounds}.
\end{proof}

\subsection{The extremal copulas}\label{sec:extremizers}

We place $\pi_\theta$ in the upper magnitude interval, couple the lower
interval along the diagonal, and assign signs according to the following rule.

\begin{definition}[The family $C_\theta$]\label{def:extremizers}
	Fix $\theta>0$, write $a=a_\theta$ and $z=z_\theta$, and take $\pi_\theta$ from Lemma~\ref{lem:input}.
	Define $(A,B,S)$ by the following mixture. With probability $a$, choose $A=B$ uniformly from $[0,a]$ and put $S=-1$. With the remaining probability $z=1-a$, sample $(U_0,V_0)\sim\pi_\theta$, set $A=a+zU_0$ and $B=a+zV_0$, and put $S=1$.
	For an independent fair sign $\varepsilon$, define $C_\theta$ as the copula of
	\[
		U=\frac{1+\varepsilon A}{2},\qquad V=\frac{1+\varepsilon SB}{2}.
	\]
	Complete the family at its endpoints by setting $C_0=W$ and $C_\infty=M$.
\end{definition}

Within either component, each magnitude marginal is uniform on the corresponding interval, and the mixture probabilities $a$ and $z$ coincide with those interval lengths. Both magnitudes are consequently uniform on $\I$, so Lemma~\ref{lem:representations} shows that $C_\theta$ is a copula.
In the lower interval, the choices $A=B$ and $S=-1$ produce the central antidiagonal segment $U+V=1$. In the upper interval, $S=1$ produces two corner blocks of mass $z/2$ each. The sign $\varepsilon$ selects the corner and simultaneously reflects both coordinates in the lower corner.

Denote the copula of $\pi_\theta$ by $B_\theta$, and write $\widehat B_\theta(u,v)=u+v-1+B_\theta(1-u,1-v)$ for its survival copula.
The copula $C_\theta$ is then the ordinal sum of $\widehat B_\theta$, $W$ and $B_\theta$ over the successive intervals $[0,z/2]$, $[z/2,1-z/2]$ and $[1-z/2,1]$, of lengths $z/2$, $a$ and $z/2$.
More explicitly, for a component $D$ assigned to $[b,b+\ell_0]$, one has $C_\theta(u,v)=b+\ell_0D((u-b)/\ell_0,(v-b)/\ell_0)$ inside that square and $C_\theta(u,v)=\min(u,v)$ whenever the coordinates belong to distinct blocks.
For the theory of ordinal sums, see \citet{durante2015principles}. Examples of their role in exact-region constructions appear in \citet{bukovsek2026exact,orendaylares2026exact}.
Figure~\ref{fig:extremizers} displays the resulting block sizes and supports.

\begin{proposition}[Moments of $C_\theta$]\label{prop:moments}
	The relations $\gamma(C_\theta)=G(\theta)$ and $\rho(C_\theta)=P(\theta)$ are valid for every $\theta\in[0,\infty]$.
\end{proposition}

\begin{proof}
	Let $\theta>0$ be fixed, and abbreviate $a=a_\theta$, $z=z_\theta$.
	For $(U_0,V_0)\sim\pi_\theta$, the moment relations \eqref{eq:input-moments} imply
	\[
		\E\min(U_0,V_0)=\frac{1-m_\theta}{2},\qquad
		\E[U_0V_0]=\frac13-\frac{q_\theta}{2}.
	\]
	Along the antidiagonal component, $A=B$ is uniform on $[0,a]$ and $S=-1$. By \eqref{eq:signed-moments}, its contributions to $\gamma$ and $\rho$ are respectively $2a\cdot(-a/2)=-a^2$ and $3a\cdot(-a^2/3)=-a^3$. The corner squares contribute
	\[
	\begin{aligned}
		2z\,\E\min(a+zU_0,a+zV_0)&=2za+z^2(1-m_\theta),\\
		3z\,\E\bigl[(a+zU_0)(a+zV_0)\bigr]&=3za^2+3z^2a+z^3-\tfrac32z^3q_\theta.
	\end{aligned}
	\]
	Because $a+z=1$, the identities $2za+z^2=1-a^2$ and $3za^2+3z^2a+z^3=1-a^3$ reduce these expressions to \eqref{eq:param-boundary}. At the two endpoints, the result follows from $C_0=W$ and $C_\infty=M$.
\end{proof}

\begin{example}[Shuffles on the elementary arc]\label{ex:map}
	Given $0<d\le d_*$, define the measure-preserving map $T_d\colon\I\to\I$ by
	\[
		T_d(u)=\begin{cases}
			u+d/2,&0\le u<d/2,\\
			u-d/2,&d/2\le u<d,\\
			1-u,&d\le u\le1-d,\\
			u+d/2,&1-d<u<1-d/2,\\
			u-d/2,&1-d/2\le u\le1,
		\end{cases}
	\]
	and denote by $C^{T_d}$ the copula of $(U,T_d(U))$ for uniform $U$. This is a five-piece shuffle of $M$, as shown in Figure~\ref{fig:extremizers}(a),(b).
	For $0<\theta\le1$, Lemma~\ref{lem:input} chooses the half shift for $\pi_\theta$. Expanding Definition~\ref{def:extremizers} and using $z_\theta=2d$ gives $C_\theta=C^{T_d}$: the middle piece is the antidiagonal segment, whereas the half shift in each corner square becomes a pair of parallel diagonal pieces, each of length $d/2$.
	As $\theta$ ranges over $(0,1]$, $z_\theta$ grows continuously from $0$ to $2d_*$, so all $d\in(0,d_*]$ arise.
	Finally, Proposition~\ref{prop:moments} together with \eqref{eq:d-arc} gives $\gamma(C^{T_d})=-1+8d-10d^2$ and $\rho(C^{T_d})=-1+12d-24d^2+13d^3$. Direct integration yields the same formulas.
\end{example}

\section{The glued dual potential and proof of exactness}\label{sec:proof}

For every supporting parameter $t_\theta$, we build a dual potential
that remains feasible throughout the magnitude square. Its equality set
contains the support of $C_\theta$, which produces the sharp support bounds.

\subsection{Gluing the potentials}\label{sec:gluing}

Take $\theta>0$ and abbreviate
$a=a_\theta$, $z=z_\theta$, $t=t_\theta$ and $h=h_\theta$.
The two magnitude blocks from Definition~\ref{def:extremizers} naturally
lead to two formulas for the potential. We obtain each formula from
contact, join them at $a$, and then establish feasibility on the full square.

Along the antidiagonal component, $A=B=x$ and $x\le a<t$. Hence contact requires $2f(x)=c_t(x,x)=x(t-x)$.
Accordingly, the potential on $[0,a]$ must be the quadratic $x(t-x)/2$.

In a corner block, the magnitudes take the form $x=a+zu$ and $y=a+zv$, where $(u,v)\sim\pi_\theta$. The relevant cost there is $xy-t\min(x,y)$. It coincides with $c_t(x,y)$ provided that $\max(x,y)\ge t$, a property checked in the proof of Theorem~\ref{thm:support}.
To connect this cost with \eqref{eq:input-dual}, suppose that $a\le x\le y\le1$ and set $u=(x-a)/z$ and $v=(y-a)/z$.
Using $2xy=x^2+y^2-(x-y)^2$, we find
\[
	xy-tx=\frac{x^2-tx}{2}+\frac{y^2-ty}{2}-\frac12\bigl\{(y-x)^2-t(y-x)\bigr\},
\]
Because $y-x=z(v-u)$ and $t=sz$, the braced expression is $z^2\{(u-v)^2-s|u-v|\}$, namely the rho--footrule cost after rescaling.
Inequality \eqref{eq:input-dual} bounds this quantity below by $z^2\{h(u)+h(v)\}$, and equality holds $\pi_\theta$-almost surely.
We are therefore led to set, for $x\in[a,1]$,
\begin{equation}\label{eq:upper-potential}
	f_+(x)=\frac{x^2-tx}{2}-\frac{z^2}{2}\,h\Bigl(\frac{x-a}{z}\Bigr),
\end{equation}
Indeed, for $a\le x\le y\le1$ and the preceding $u,v$,
\begin{equation}\label{eq:positive-branch}
	f_+(x)+f_+(y)-(xy-tx)=\frac{z^2}{2}\Bigl\{(u-v)^2-s|u-v|-h(u)-h(v)\Bigr\}\ge0,
\end{equation}
and equality occurs for $\pi_\theta$-almost every $(u,v)$.

The definition of $\alpha_\theta$ in \eqref{eq:split} ensures that the two formulas coincide at $x=a$.
In fact, $\alpha_\theta^2-\alpha_\theta=\theta^2c_\theta/2$, $a=\alpha_\theta t$, $z=\theta t$, and $h(0)=c_\theta$, and consequently
\begin{equation}\label{eq:potential-match}
	f_+(a)=\frac{t^2}{2}\Bigl(\alpha_\theta^2-\alpha_\theta-\theta^2c_\theta\Bigr)=-\frac{z^2c_\theta}{4}=\frac{a(t-a)}{2}>0,
\end{equation}
The final term is precisely the value at $a$ of the quadratic branch. Lemma~\ref{lem:parameter-checks} gives $c_\theta<0$, so this common value is positive.

We next show that $f_+$ does not decrease.
Since \eqref{eq:input-endpoint} gives $|h'|\le w_\theta$ almost everywhere, almost every $x\in[a,1]$ satisfies
\begin{equation}\label{eq:potential-monotone}
	f_+'(x)=x-\frac t2-\frac z2\,h'\Bigl(\frac{x-a}{z}\Bigr)\ge a-\frac{t+zw_\theta}{2}=t\Bigl(\alpha_\theta-\frac{1+\theta w_\theta}{2}\Bigr)\ge0,
\end{equation}
with the final inequality following from Lemma~\ref{lem:parameter-checks}.
Thus $f_+$ is nondecreasing on $[a,1]$ and, in view of \eqref{eq:potential-match}, positive throughout that interval.

Estimate \eqref{eq:positive-branch} applies in the upper block whenever
the larger magnitude is no smaller than $t$. Full feasibility additionally
requires the other branch of the absolute value and pairs lying on opposite
sides of the split. The lemma below covers each of the three possible
placements of the two magnitudes.

\begin{lemma}[Dual feasibility of the glued potential]\label{lem:extension}
	Define
	\begin{equation}\label{eq:glued-potential}
		f_\theta(x)=\begin{cases}x(t-x)/2,&0\le x\le a,\\ f_+(x),&a\le x\le1.\end{cases}
	\end{equation}
	The function $f_\theta$ is continuous and satisfies
	\begin{equation}\label{eq:global-dual}
		f_\theta(x)+f_\theta(y)\ge c_t(x,y)=\min(x,y)\,|\max(x,y)-t|\qquad\text{for all }(x,y)\in\I^2,
	\end{equation}
	Thus $f_\theta$ is dual feasible in \eqref{eq:dual} for $t=t_\theta$.
\end{lemma}

\begin{proof}
	Equation \eqref{eq:potential-match} proves continuity at $a$.
	The two sides of \eqref{eq:global-dual} are symmetric in $(x,y)$. We may therefore take $x\le y$, so that $c_t(x,y)=x\,|y-t|$, and separate the argument according to how $x$ and $y$ lie with respect to $a$.
	\begin{enumerate}[(i)]
		\item\emph{$y\le a$.}
		Here both magnitudes belong to the quadratic branch. Since \eqref{eq:split-bounds} yields $y\le a<t$, we have $c_t(x,y)=x(t-y)$.
		Calculation then gives
		\[
			f_\theta(x)+f_\theta(y)-x(t-y)=\tfrac12(y-x)\bigl\{t-(y-x)\bigr\},
		\]
		This is nonnegative because $0\le y-x\le a<t$.
		\item\emph{$a\le x\le y$.}
		Both magnitudes now fall in the upper branch, for which \eqref{eq:positive-branch} yields $f_+(x)+f_+(y)\ge xy-tx$.
		When $y\ge t$, the identity $c_t(x,y)=xy-tx$ proves the claim.
		When $y<t$, instead $c_t(x,y)=x(t-y)$, and comparison with the joining value at $a$ gives
		\[
			c_t(x,y)=x(t-y)\le x(t-x)\le a(t-a)=2f_+(a)\le f_+(x)+f_+(y).
		\]
		The first inequality follows from $y\ge x$. For the second, observe that $x\mapsto x(t-x)$ decreases on $[t/2,t]$, an interval containing $a$ and $x$ by \eqref{eq:split-bounds}. Equation \eqref{eq:potential-match} supplies the equality, and monotonicity of $f_+$ gives the final inequality.
		\item\emph{$x\le a\le y$.}
		With $y$ fixed, define
		\[
			\psi(x)=\frac{x(t-x)}{2}+f_+(y)-x\,|y-t|,\qquad0\le x\le a,
		\]
		On this interval, $\psi(x)$ equals $f_\theta(x)+f_\theta(y)-c_t(x,y)$.
		Concavity of $\psi$ reduces the verification to its endpoints.
		At the left endpoint, $\psi(0)=f_+(y)\ge0$.
		At the right endpoint, \eqref{eq:potential-match} gives $\psi(a)=f_+(a)+f_+(y)-a\,|y-t|$, and case~(ii), applied to $(a,y)$, shows that this value is nonnegative.
		It follows that $\psi\ge0$ throughout $[0,a]$.\qedhere
	\end{enumerate}
\end{proof}

\begin{theorem}[Sharp supporting inequalities]\label{thm:support}
	Let $\theta>0$ and let $\mu_\theta\in\couplings$ be the law of the magnitudes $(A,B)$ in Definition~\ref{def:extremizers}.
	The measure $\mu_\theta$ is supported by the contact set $\Gamma_{f_\theta}$. Consequently, $\mu_\theta$ solves \eqref{eq:primal} and $f_\theta$ solves \eqref{eq:dual} at $t=t_\theta$, and
	\begin{equation}\label{eq:value}
		\mathsf V(t_\theta)=\mathsf D(t_\theta)=2\int_0^1f_\theta(x)\de x=\tfrac13\bigl\{P(\theta)-\tfrac32t_\theta\,G(\theta)\bigr\}.
	\end{equation}
	It follows that every $C\in\CC$ obeys
	\begin{equation}\label{eq:sharp-support}
		\rho(C)-\tfrac32t_\theta\,\gamma(C)\le P(\theta)-\tfrac32t_\theta\,G(\theta),
	\end{equation}
	Equality is attained by $C=C_\theta$.
\end{theorem}

\begin{proof}
	Lemma~\ref{lem:extension} supplies dual feasibility of $f_\theta$.
	For equality in \eqref{eq:sharp-support}, we need both contact in the magnitude problem and the optimizing sign from \eqref{eq:signed-objective}. We check these properties on each component of the coupling.

	\emph{Contact on the antidiagonal part.}
	On this component, $A=B\le a<t$. Hence $2f_\theta(A)=A(t-A)=c_t(A,A)$, and therefore $(A,A)\in\Gamma_{f_\theta}$.

	\emph{Contact on the corner part.}
	On a corner component, $(A,B)=(a+zU_0,a+zV_0)$ for $(U_0,V_0)\sim\pi_\theta$, and equality holds in \eqref{eq:positive-branch}. Thus
	\[
		f_\theta(A)+f_\theta(B)=AB-t\min(A,B)=\min(A,B)\{\max(A,B)-t\}.
	\]
	Equations \eqref{eq:potential-monotone} and \eqref{eq:potential-match} show that $f_+$ is positive on $[a,1]$, so the left-hand side is positive.
	Because $\min(A,B)\ge a>0$, we must have $\max(A,B)>t$. The right-hand side is therefore $c_t(A,B)$, and $(A,B)\in\Gamma_{f_\theta}$.
	Combined with antidiagonal contact, this proves $\mu_\theta(\Gamma_{f_\theta})=1$. Lemma~\ref{lem:duality}(ii) now establishes optimality of $\mu_\theta$ and $f_\theta$, together with the first two identities in \eqref{eq:value}.

	\emph{Optimal sign.}
	Definition~\ref{def:extremizers} assigns $S=-1$ to the antidiagonal component, where $\max(A,B)<t$, and $S=1$ to the corner component, where we have just proved $\max(A,B)>t$.
	Consequently $S=\sgn(\max(A,B)-t)$ almost surely, which makes \eqref{eq:signed-objective} an equality for $C_\theta$.

	\emph{The value.}
	Equality in \eqref{eq:signed-objective}, together with Proposition~\ref{prop:moments}, gives
	\[
		3\int_{\I^2}c_t(x,y)\de\mu_\theta(x,y)=\rho(C_\theta)-\tfrac32t\,\gamma(C_\theta)=P(\theta)-\tfrac32t\,G(\theta),
	\]
	which proves the last identity in \eqref{eq:value}.
	Combining \eqref{eq:support-function} with \eqref{eq:value} now yields \eqref{eq:sharp-support}. Equality for $C=C_\theta$ was established above.
\end{proof}

\subsection{Proof of Theorem~\ref{thm:main}}\label{sec:completeness}

\begin{proof}[Proof of Theorem~\ref{thm:main}]
	\emph{The maximum \eqref{eq:maximum}.}
	Suppose that $0<\theta<\infty$ and $C\in\CC$ has $\gamma(C)=G(\theta)$.
	The gamma terms cancel in \eqref{eq:sharp-support}, leaving $\rho(C)\le P(\theta)$. Proposition~\ref{prop:moments} shows that $C_\theta$ realizes this bound at the required gamma value.

	We still have to consider $\theta\in\{0,\infty\}$. At these endpoints, the corresponding sections of $\Om$ each contain a single point.
	Indeed, the bounds $C(u,u)\le u$ and $C(u,1-u)\le\min(u,1-u)$ imply $\gamma(C)\le1$ through \eqref{eq:defs}. Equality requires $C(u,u)=u$ for almost every $u\in\I$, and hence for every $u$ by continuity.
	Under this condition,
	\[
		\PP(U\le u<V)=u-C(u,u)=0\qquad\text{for all }u\in\I,
	\]
	whence $\PP(U<V)=0$. By symmetry, $\PP(U>V)=0$ as well, and therefore $C=M$.
	We conclude that $\gamma(C)=1$ exactly when $C=M$. The reflection argument given below likewise shows that $\gamma(C)=-1$ exactly when $C=W$.
	This establishes \eqref{eq:maximum} also at $\theta\in\{0,\infty\}$.

	\emph{Continuity and surjectivity.}
	On each parameter interval, continuity follows from \eqref{eq:small-parameters}--\eqref{eq:param-boundary}. Lemma~\ref{lem:parameter-checks} guarantees positivity of the radicand in \eqref{eq:split}.
	It remains to verify compatibility at the three types of junction.
	When $\theta=N$, both neighboring formulas have $\ell=1/(2N)$ and $\delta=0$, so their values coincide.
	When $\theta=1/(L+R)$, either choice of $\ell$ yields $p=0$ and the same quantities
	\[
		m_\theta=\frac{L+R}{2},\qquad c_\theta=-\frac{LR}{2},\qquad q_\theta=LR+\tfrac23N(N+1)(L-R)^3.
	\]
	Finally, at $\theta=1$, the piecewise formulas with $N=1$, $\ell=L=\tfrac12$, and $\delta=0$ yield
	\[
		m_1=\tfrac12,\qquad q_1=\tfrac14,\qquad c_1=-\tfrac18,
	\]
	which matches \eqref{eq:large-parameters}.

	As $\theta\downarrow0$, relation \eqref{eq:large-parameters} implies $\theta^2c_\theta\to0$, and therefore $\alpha_\theta\to1$, $a_\theta\to1$, and $z_\theta\to0$.
	As $\theta\to\infty$, both $m_\theta$ and $q_\theta$ vanish, whereas Lemma~\ref{lem:parameter-checks} keeps $\alpha_\theta$ bounded. It follows that $a_\theta\to0$ and $z_\theta\to1$.
	We thus obtain $(G(0),P(0))=(-1,-1)$ and $(G(\infty),P(\infty))=(1,1)$, together with continuity on $[0,\infty]$.
	The intermediate value theorem then shows that $G([0,\infty])=[-1,1]$.
	If two parameter values produce the same $G$, equation \eqref{eq:maximum} identifies both corresponding $P$ values with the maximum of $\rho$ on one and the same section of $\Om$. Hence $\rhoup$ is well defined.

	\emph{The region.}
	Given $(U,V)\sim C$, let $C^-(u,v)=u-C(u,1-v)$ denote the copula of $(U,1-V)$. Then
	\[
		\rho(C^-)=12\,\E[U(1-V)]-3=-\rho(C),\qquad\gamma(C^-)=-\gamma(C),
	\]
	The second equality follows from \eqref{eq:rank-moments}, because $U+(1-V)-1=U-V$ and $U-(1-V)=U+V-1$.
	Thus $\Om$ is centrally symmetric, and its lower boundary equals $-\rhoup(-g)$.
	Since both functionals in \eqref{eq:defs} depend affinely on $C$, mixtures of the two extremizers at fixed $g$ attain every value between the upper and lower boundaries.
	This proves \eqref{eq:exact-region}.

	Under uniform convergence, $\CC$ is compact by \citet[Sec.~1.7]{durante2015principles}, while \eqref{eq:defs} is continuous and affine in $C$. Therefore $\Om$ is both compact and convex.
	Convexity of $\Om$ makes $\rhoup$ concave, which implies continuity on $(-1,1)$.
	Continuity also holds at the endpoints: by compactness and the singleton sections at $g=\pm1$, every accumulation point of $(\rhoup(g_n),g_n)$ for $g_n\to\pm1$ belongs to the relevant endpoint section.

	It remains to prove strict monotonicity of $\rhoup$.
	Choose $-1\le g_1<g_2\le1$, and let $C_1$ attain the boundary point $(\rhoup(g_1),g_1)$.
	We have $\rho(C_1)<1$: by \eqref{eq:rank-moments}, the equality $\rho(C)=1$ implies $U=V$ almost surely, hence $C=M$, whereas $\gamma(M)=1>g_1$.
	Set $\eta=(g_2-g_1)/(1-g_1)\in(0,1]$ and $C_\eta=(1-\eta)C_1+\eta M$. Then
	\[
		\gamma(C_\eta)=(1-\eta)g_1+\eta=g_2,\qquad
		\rho(C_\eta)=(1-\eta)\rhoup(g_1)+\eta>\rhoup(g_1),
	\]
	which forces $\rhoup(g_2)>\rhoup(g_1)$.
\end{proof}

\subsection{Proofs of the consequences}\label{sec:consequence-proofs}

\begin{proof}[Proof of Proposition~\ref{prop:discrepancy}]
	In \eqref{eq:sharp-support}, the objective $\rho-\gamma$ corresponds to the slope $\tfrac32t=1$, or equivalently $t=\tfrac23$.
	We begin by locating the elementary-arc parameter $\theta_0$ for which $t_{\theta_0}=\tfrac23$.
	Put $\theta_0=3d_0\in(0,1)$ and $\alpha=\tfrac32-3d_0$.
	Equation \eqref{eq:discrepancy-constant} gives $39d_0^2-28d_0+4=0$, from which direct calculation yields
	\[
		\alpha^2-\alpha=\frac{3\theta_0^2}{16}-\frac{\theta_0}{4}=\frac{\theta_0^2c_{\theta_0}}{2}\qquad\text{and}\qquad\alpha>\frac12.
	\]
	Since $\alpha>\tfrac12$, equation \eqref{eq:split} identifies it with $\alpha_{\theta_0}$. Consequently $t_{\theta_0}=1/(\theta_0+\alpha)=\tfrac23$ and $z_{\theta_0}=\theta_0t_{\theta_0}=2d_0$.

	Applying Theorem~\ref{thm:support} at $\theta=\theta_0$ gives $\rho(C)-\gamma(C)\le P(\theta_0)-G(\theta_0)$ for every $C\in\CC$. Example~\ref{ex:map} shows that equality is attained by $C_{\theta_0}=C^{T_{d_0}}$.
	Formula \eqref{eq:d-arc} gives
	\[
		P(\theta_0)-G(\theta_0)=4d_0-14d_0^2+13d_0^3=\Delta_0,
	\]
	while substitution of $d=d_0$ in \eqref{eq:d-arc} supplies the maximizer's coordinates.
	Reflecting the copula gives the identical bound for $\gamma-\rho$, completing the proof of \eqref{eq:discrepancy}.
\end{proof}

\begin{proof}[Proof of Corollary~\ref{cor:signs}]
	Equation \eqref{eq:exact-region} identifies the section of $\Om$ at $\gamma=0$ as $[-\rhoup(0),\rhoup(0)]$. Since \eqref{eq:elementary} gives $\rhoup(0)=r_0$, the first identity in \eqref{eq:zero-sections} follows.

	For the second identity, strict increase of $\rhoup$, together with $\rhoup(-1)=-1<0<\rhoup(0)$, gives a unique zero $g_z\in(-1,0)$. The section at $\rho=0$ consists of those $g$ satisfying $-\rhoup(-g)\le0\le\rhoup(g)$ and is therefore $[g_z,-g_z]$.
	Along the elementary arc \eqref{eq:d-arc}, the condition $r=0$ becomes
	\[
		13d^3-24d^2+12d-1=(d-1)(13d^2-11d+1)=0,
	\]
	Its unique root in $[0,d_*]$ is
	\[
		d_1=\frac{11-\sqrt{69}}{26}=0.1035\ldots,\qquad\text{with}\qquad -1+8d_1-10d_1^2=-g_0
	\]
	where $g_0$ is the constant in \eqref{eq:thresholds}. Thus $g_z=-g_0$, proving the second identity in \eqref{eq:zero-sections}.
	Monotonicity of $\rhoup$ gives the stated implications, and attainment of $(\pm r_0,0)$ and $(0,\pm g_0)$ shows that both thresholds are sharp.
\end{proof}

\bibliographystyle{plainnat}
\bibliography{\rhogammaroot references}

\begin{thebibliography}{14}
\providecommand{\natexlab}[1]{#1}
\providecommand{\url}[1]{\texttt{#1}}
\expandafter\ifx\csname urlstyle\endcsname\relax
  \providecommand{\doi}[1]{doi: #1}\else
  \providecommand{\doi}{doi: \begingroup \urlstyle{rm}\Url}\fi

\bibitem[Ansari and Rockel(2026)]{ansari2026exact}
Jonathan Ansari and Marcus Rockel.
\newblock The exact {Spearman} rho--footrule region via optimal transport with
  applications to finite rankings, mixability, and {Chatterjee}'s rank
  correlation.
\newblock Preprint,
  \href{https://arxiv.org/abs/2608.20176v1}{arXiv:2608.20176v1}, 2026.

\bibitem[Durante and Sempi(2015)]{durante2015principles}
Fabrizio Durante and Carlo Sempi.
\newblock \emph{Principles of Copula Theory}.
\newblock Chapman and Hall/CRC, Boca Raton, 2015.
\newblock \doi{10.1201/b18674}.

\bibitem[Genest et~al.(2010)Genest, Ne{\v{s}}lehov{\'a}, and
  Ben~Ghorbal]{genest2010spearman}
Christian Genest, Johanna Ne{\v{s}}lehov{\'a}, and Noomen Ben~Ghorbal.
\newblock {Spearman}'s footrule and {Gini}'s gamma: a review with complements.
\newblock \emph{J. Nonparametric Stat.}, 22\penalty0 (8):\penalty0 937--954,
  2010.
\newblock \doi{10.1080/10485250903499667}.

\bibitem[Kokol~Bukov{\v{s}}ek and Moj{\v{s}}kerc(2022)]{bukovsek2022exact}
Damjana Kokol~Bukov{\v{s}}ek and Bla{\v{z}} Moj{\v{s}}kerc.
\newblock On the exact region determined by {Spearman}'s footrule and {Gini}'s
  gamma.
\newblock \emph{J. Comput. Appl. Math.}, 410:\penalty0 114212, 2022.
\newblock \doi{10.1016/j.cam.2022.114212}.

\bibitem[Kokol~Bukov{\v{s}}ek and Stopar(2023)]{bukovsek2023exact}
Damjana Kokol~Bukov{\v{s}}ek and Nik Stopar.
\newblock On the exact regions determined by {Kendall}'s tau and other
  concordance measures.
\newblock \emph{Mediterr. J. Math.}, 20\penalty0 (3):\penalty0 147, 2023.
\newblock \doi{10.1007/s00009-023-02350-0}.

\bibitem[Kokol~Bukov{\v{s}}ek et~al.(2021)Kokol~Bukov{\v{s}}ek, Ko{\v{s}}ir,
  Moj{\v{s}}kerc, and Omladi{\v{c}}]{bukovsek2021spearman}
Damjana Kokol~Bukov{\v{s}}ek, Toma{\v{z}} Ko{\v{s}}ir, Bla{\v{z}}
  Moj{\v{s}}kerc, and Matja{\v{z}} Omladi{\v{c}}.
\newblock {Spearman}'s footrule and {Gini}'s gamma: local bounds for bivariate
  copulas and the exact region with respect to {Blomqvist}'s beta.
\newblock \emph{J. Comput. Appl. Math.}, 390:\penalty0 113385, 2021.
\newblock \doi{10.1016/j.cam.2021.113385}.

\bibitem[Kokol~Bukov{\v{s}}ek et~al.(2026)Kokol~Bukov{\v{s}}ek, Lazi{\'c},
  Moj{\v{s}}kerc, and Stopar]{bukovsek2026exact}
Damjana Kokol~Bukov{\v{s}}ek, Petra Lazi{\'c}, Bla{\v{z}} Moj{\v{s}}kerc, and
  Nik Stopar.
\newblock The exact region determined by {Spearman}'s footrule, {Gini}'s gamma
  and {Kendall}'s tau.
\newblock \emph{J. Comput. Appl. Math.}, 486:\penalty0 117639, 2026.
\newblock \doi{10.1016/j.cam.2026.117639}.

\bibitem[Nelsen(1998)]{nelsen1998concordance}
Roger~B. Nelsen.
\newblock Concordance and {Gini}'s measure of association.
\newblock \emph{J. Nonparametric Stat.}, 9\penalty0 (3):\penalty0 227--238,
  1998.
\newblock \doi{10.1080/10485259808832744}.

\bibitem[Nelsen(2006)]{nelsen2006introduction}
Roger~B. Nelsen.
\newblock \emph{An Introduction to Copulas}.
\newblock Springer Series in Statistics. Springer, New York, second edition,
  2006.
\newblock \doi{10.1007/0-387-28678-0}.

\bibitem[Orenday~Lares and Rockel(2026)]{orendaylares2026exact}
Jacob~Israel Orenday~Lares and Marcus Rockel.
\newblock The exact region determined by {Kendall}'s tau, {Spearman}'s footrule
  and {Blomqvist}'s beta.
\newblock Preprint, \href{https://arxiv.org/abs/2607.12841}{arXiv:2607.12841},
  2026.

\bibitem[Scarsini(1984)]{scarsini1984measures}
Marco Scarsini.
\newblock On measures of concordance.
\newblock \emph{Stochastica}, 8\penalty0 (3):\penalty0 201--218, 1984.
\newblock URL \url{https://eudml.org/doc/38916}.

\bibitem[Schreyer et~al.(2017)Schreyer, Paulin, and
  Trutschnig]{schreyer2017exact}
Manuela Schreyer, Roland Paulin, and Wolfgang Trutschnig.
\newblock On the exact region determined by {Kendall}'s $\tau$ and {Spearman}'s
  $\rho$.
\newblock \emph{J. R. Stat. Soc., Ser. B, Stat. Methodol.}, 79\penalty0
  (2):\penalty0 613--633, 2017.
\newblock \doi{10.1111/rssb.12181}.

\bibitem[Villani(2009)]{villani2009optimal}
C{\'e}dric Villani.
\newblock \emph{Optimal Transport: Old and New}, volume 338 of
  \emph{Grundlehren der mathematischen Wissenschaften}.
\newblock Springer, Berlin, 2009.
\newblock \doi{10.1007/978-3-540-71050-9}.

\bibitem[Zaev(2015)]{zaev2015monge}
Danila~A. Zaev.
\newblock On the {Monge}--{Kantorovich} problem with additional linear
  constraints.
\newblock \emph{Math. Notes}, 98\penalty0 (5):\penalty0 725--741, 2015.
\newblock \doi{10.1134/S0001434615110036}.

\end{thebibliography}
\end{document}